\documentclass[11pt]{amsart}
\usepackage{float}
\usepackage[usenames]{color}
\usepackage{graphicx}
\usepackage{amssymb}
\usepackage{amscd}
\usepackage{makecell}
\theoremstyle{plain}
\newtheorem{thm}{Theorem}[section]
\newtheorem{lemma}[thm]{Lemma}

\newtheorem{prop}[thm]{Proposition}

\theoremstyle{definition}

\newtheorem{rmk}[thm]{Remark}
\newtheorem{example}[thm]{Example}

\def\dim{\mathop{\hbox {dim}}\nolimits}

\newcommand{\fra}{\mathfrak{a}}
\newcommand{\frb}{\mathfrak{b}}

\newcommand{\frg}{\mathfrak{g}}
\newcommand{\frh}{\mathfrak{h}}

\newcommand{\frk}{\mathfrak{k}}
\newcommand{\frl}{\mathfrak{l}}

\newcommand{\frp}{\mathfrak{p}}
\newcommand{\frq}{\mathfrak{q}}

\newcommand{\frt}{\mathfrak{t}}
\newcommand{\fru}{\mathfrak{u}}

\newcommand{\bbC}{\mathbb{C}}

\newcommand{\bbR}{\mathbb{R}}

\newcommand{\bbZ}{\mathbb{Z}}

\newcommand{\caC}{\mathcal{C}}

\newcommand{\caL}{\mathcal{L}}

\newcommand{\caR}{\mathcal{R}}

\begin{document}

\title{Dirac series of $E_{7(-25)}$}

\author{Yi-Hao Ding}
\address[Ding]{School of Mathematical Sciences, Soochow University, Suzhou 215006,
P.~R.~China}
\email{435025738@qq.com}

\author{Chao-Ping Dong}
\address[Dong]{School of Mathematical Sciences, Soochow University, Suzhou 215006,
P.~R.~China}
\email{chaopindong@163.com}


\abstract{By further sharpening the Helgason-Johnson bound in 1969, this paper classifies all the irreducible unitary representations with non-zero Dirac cohomology of the Hermitian symmetric real form $E_{7(-25)}$. }

\endabstract

\subjclass[2010]{Primary 22E46}

\keywords{Dirac cohomology, non-vanishing, number of strings}

\maketitle


\section{Introduction}

As a sequel to \cite{DDH} and \cite{DDY}, this article aims to classify the irreducible unitary representations with non-zero Dirac cohomology for the simple linear Lie group $E_{7(-25)}$, which is Hermitian symmetric.

We embark with a complex connected  simple algebraic group $G_{\bbC}$ which has finite center.
Let $\sigma: G_{\bbC} \to G_{\bbC}$ be a \emph{real form} of $G_{\bbC}$. That is, $\sigma$ is an antiholomorphic Lie group automorphism and $\sigma^2={\rm Id}$. Let $\theta: G_{\bbC}\to G_{\bbC}$ be the involutive algebraic automorphism of $G_{\bbC}$ corresponding to $\sigma$ via Cartan theorem (see Theorem 3.2 of the paper \cite{ALTV} by Adams, van Leeuwen, Trapa and Vogan). Put $G=G_{\bbC}^{\sigma}$ as the group of real points. Note that $G$ must be in the Harish-Chandra class \cite{HC}.  Denote by $K_{\bbC}:=G_{\bbC}^{\theta}$, and put $K:=K_{\bbC}^{\sigma}$. Denote by $\frg_0$ the Lie algebra of $G$, and let
$$
\frg_0=\frk_0\oplus \frp_0
$$
be the  Cartan decomposition corresponding to $\theta$ on the Lie algebra level.

Let $H_f=T_f A_f$ be the $\theta$-stable fundamental Cartan subgroup for $G$. Then  $T_f$ is a maximal torus of $K$, and on the Lie algebra level,
$$
\frh_{f, 0}=\frt_{f, 0}\oplus \fra_{f, 0}
$$
is the unique $\theta$-stable fundamental Cartan subalgebra of $\frg_0$.  As usual, the subscripts are dropped to stand for the complexified Lie algebras. For example, $\frg=\frg_0\otimes_{\bbR}\bbC$, $\frh_f=\frh_{f, 0}\otimes_{\bbR}\bbC$ and so on. We fix a non-degenerate invariant symmetric bilinear form $B(\cdot, \cdot)$ on $\frg$. Its restrictions to $\frk$, $\frp$, etc., will also be denoted by the same symbol.

Let $l$ be the rank of $\frg$. That is, $l=\dim_{\bbC} \frh$. Fix a positive root system $\Delta^+(\frg, \frh_f)$, and let $\{\zeta_1, \dots, \zeta_l\}$ be the corresponding fundamental weights. Restricting the roots in $\Delta^+(\frg, \frh_f)$ to $\frt_f$, we have that
$$
\Delta^+(\frg, \frt_f)=\Delta^+(\frk, \frt_f) \cup \Delta^+(\frp, \frt_f).
$$
Let $\rho$ (resp., $\rho_c$) be the half sum of roots in $\Delta^+(\frg, \frt_f)$ (resp., $\Delta^+(\frk, \frt_f)$). Then $\rho_n:=\rho-\rho_c$
is the half sum of roots in $\Delta^+(\frp, \frt_f)$.
We will denote the Weyl groups corresponding to these root systems by $W(\frg, \frh_f)$, $W(\frg, \frt_f)$, $W(\frk, \frt_f)$.

Fix an orthonormal basis $Z_1, \dots, Z_n$ of $\frp_0$ with respect to
the inner product induced by the form $B(\cdot, \cdot)$. Let $U(\frg)$ be the
universal enveloping algebra of $\frg$ and let $C(\frp)$ be the
Clifford algebra of $\frp$ with respect to $B(\cdot, \cdot)$. In 1972,
Parthasarathy introduced the \emph{Dirac operator} as
\begin{equation}
D:=\sum_{i=1}^{n}\, Z_i \otimes Z_i\in U(\frg)\otimes C(\frp).
\end{equation}
It is easy to check that $D$ is independent of the choice of the
orthonormal basis $Z_i$ and it is $K$-invariant under the diagonal
action of $K$ given by adjoint actions on both factors. Moreover, we embed the Lie algebra $\frk$ diagonally into $U(\frg)\otimes C(\frp)$ in the following way:
$$
X\mapsto X_{\Delta}:=X\otimes 1 + 1 \otimes \sum_{i<j}\frac{1}{2}B(X, [Z_i, Z_j])Z_i Z_j, \quad X\in\frk.
$$
Denote the image of $\frk$ by $\frk_{\Delta}$. Let $\Omega_{\frg}$ (resp., $\Omega_{\frk_{\Delta}}$) be the  Casimir element of $\frg$ (resp., $\frk_{\Delta}$). Then
\begin{equation}\label{D-square}
D^2=-\Omega_{\frg}\otimes 1 +\Omega_{\frk_{\Delta}} +(\|\rho_c\|^2-\|\rho\|^2)1\otimes 1.
\end{equation}

Let $\pi$ be an irreducible $(\frg, K)$ module. Let $S_G$ be a spin module of $C(\frp)$. The Dirac operator acts on $\pi\otimes S_G$. When $\pi$ is unitary, the operator $D$ is self-adjoint with respect to a natural Hermitian inner product on $\pi\otimes S_G$, and $D^2$ has non-negative eigenvalue on any $\widetilde{K}$-type of $\pi\otimes S_G$.
Here $\widetilde{K}$ is the subgroup of $K\times \text{Pin}\,\frp_0$ consisting of all pairs $(k, s)$ such that $\text{Ad}(k)=p(s)$, where $\text{Ad}: K\rightarrow \text{O}(\frp_0)$ is the adjoint action, and $p: \text{Pin}\,\frp_0\rightarrow \text{O}(\frp_0)$ is the pin double covering map. Namely, $\widetilde{K}$ is constructed from the following diagram:
\[
\begin{CD}
\widetilde{K} @>  >  > {\rm Pin}\, \frp_0 \\
@VVV  @VVpV \\
K @>{\rm Ad}>> {\rm O}(\frp_0)
\end{CD}
\]
Using \eqref{D-square}, one can deduce that
\begin{equation}\label{Dirac-inequality}
\|\gamma+\rho_c\|\geq \|\Lambda\|,
\end{equation}
where $\gamma$ is a highest weight of any $\widetilde{K}$-type occurring in $\pi\otimes S_G$.
This is Parthasarathy's \emph{Dirac operator inequality} \cite{Pa2}, which effectively detects non-unitarity.
For instance, it is an important tool in the classification of irreducible unitary highest weight modules for Hermitian symmetric real forms by Enright, Howe and Wallach \cite{EHW}. However, it is \emph{not} sufficient for unitarity.

To sharpen the Dirac operator inequality, and to understand the unitary dual $\widehat{G}$ better, Vogan formulated \emph{Dirac cohomology} in 1997 \cite{Vog97} as the following $\widetilde{K}$-module:
\begin{equation}\label{def-Dirac-cohomology}
H_D(\pi)=\text{Ker}\, D/ (\text{Im} \, D \cap \text{Ker} D).
\end{equation}
Here we note that $\widetilde{K}$ acts  on $\pi$
through $K$ and on $S_G$ through the pin group
$\text{Pin}\,{\frp_0}$. Moreover, since ${\rm Ad}(k) (Z_1), \dots, {\rm Ad}(k) (Z_n)$ is still an orthonormal basis of $\frp_0$, it follows that $D$ is $\widetilde{K}$ invariant. Therefore, ${\rm  Ker} D$, ${\rm Im} D$, and $H_D(X)$ are all $\widetilde{K}$ modules.

A fundamental result pertaining to Dirac cohomology is the  Vogan conjecture proven by Huang and Pand\v zi\'c in 2002. See Theorem 2.3 of \cite{HP}. Let us state its tiny extension to possibly disconnected groups.
By setting the linear functionals on $\frt_f$ to be zero on $\fra_f$, we regard $\frt_f^{*}$ as a subspace of $\frh_f^{*}$.

\begin{thm}{\rm (Theorem A of \cite{DH})}\label{thm-HP}
Let $\pi$ be an irreducible ($\frg$, $K$)-module with infinitesimal character $\Lambda$.
Assume that the Dirac
cohomology of $\pi$ is nonzero, and let $\gamma\in\frt_f^{*}\subset\frh_f^{*}$ be any highest weight of a $\widetilde{K}$-type  in $H_D(\pi)$. Then $\Lambda$ is conjugate to
$\gamma+\rho_{c}$ under the action of the Weyl group $W(\frg,\frh_f)$.
\end{thm}

We care the most about the case that $\pi$ is unitary. Then  $\text{Ker} D \cap \text{Im} D=0$, and
\begin{equation}\label{Dirac-unitary}
H_D(\pi)=\text{Ker}\, D=\text{Ker}\, D^2.
\end{equation}
Moreover, by \cite[Theorem 3.5.2]{HP2}, the inequality \eqref{Dirac-inequality} becomes equality for certain $\widetilde{K}$-types of $\pi\otimes S_G$ if and only if $H_D(\pi)$ is non-vanishing.

Since Dirac cohomology is an invariant for Lie group representations, it is natural to ask: could we classify $\widehat{G}^d$---the set consisting of all the members of $\widehat{G}$ having non-vanishing Dirac cohomology? For convenience, we call members of $\widehat{G}^d$ the \emph{Dirac series} of $G$ (terminology suggested by J.-S. Huang). In view of the discussion after \eqref{Dirac-unitary},  the Dirac series of $G$  are exactly the members of $\widehat{G}$ on which Dirac inequality becomes equality.

The current paper aims to classify the Dirac series for $E_{7(-25)}$, by which we actually mean the connected simple real exceptional Lie  group $\texttt{E7\_h}$ in \texttt{atlas}, whose Lie algebra is denoted by ${\rm EVII}$ in Knapp \cite{Kn}, or by $E_{7(-25)}$ in other literature. Here \texttt{atlas} \cite{At} is a software which computes various types of questions relevant to the representations of $G$. For instance, it detects whether $\pi$ is unitary or not based on the algorithm in \cite{ALTV}. See Section \ref{sec-atlas} for a very brief account of \texttt{atlas}.

Let $\pi$ be any Dirac series representation of $G$ which has final \texttt{atlas} parameter $p=(x, \lambda, \nu)$. Then $\pi$ is a \textbf{FS-scattered} representation if the \textbf{KGB element} $x$ (see Section \ref{sec-atlas}) is fully supported, i.e., if $\texttt{support} (x)$ equals $[0, 1, \dots, l-1]$. Otherwise, $\pi$ will be merged into a string of representations. Theorem A of \cite{D17}  says that $\widehat{G}^d$ consists of finitely many FS-scattered representations and finitely many strings of representations. One can pin down them explicitly without reference to the entire unitary dual $\widehat{G}$. Indeed, it suffices to study the irreducible representations whose infinitesimal characters are relatively small. By further sharpening the Helason-Johnson bound \cite{HJ} on the norm of $\nu$ for $E_{7(-25)}$ (see Proposition \ref{prop-HJ}(c)), we are able to reduce the classification workload considerably. Besides the Dirac inequality \eqref{Dirac-inequality}, the \emph{unitarily small} (\emph{u-small} for short henceforth) \emph{convex hull} introduced by Salamanca-Riba and Vogan \cite{SV} is also vital to Proposition \ref{prop-HJ}(c). This concept and spin-lowest $K$-type will be recalled in Section \ref{sec-HJ}.

Now we are able to state the main result.

\begin{thm}\label{thm-EVII}
The set $\widehat{E_{7(-25)}}^d$ consists of $74$ FS-scattered representations whose spin-lowest $K$-types are all u-small, and $878$ strings of representations. Moreover, each spin lowest $K$-type of any Dirac series representation of  $E_{7(-25)}$ occurs exactly once.
\end{thm}

The paper is organized as follows: Section 2 prepares necessary material on cohomological induction and the software \texttt{atlas}. Section 3 reviews the structure of $E_{7(-25)}$. Section 4 further improves the Helgason-Johnson bound. After these preparations, Section 5 pins down the Dirac series of $E_{7(-25)}$. Section 6 aims to look at  certain Dirac series representations more carefully. Section \ref{sec-unip} considers some special unipotent representations of $E_{7(-25)}$, while Section \ref{sec-appendix} lists all its non-trivial FS-scattered Dirac series representations.

\section{Preliminaries}\label{sec-pre}

This section aims to collect necessary preliminaries.

\subsection{Cohomological induction}
Fix an element $H\in i\frt_{f,0}$. Let $\frl$ be the zero eigenspace of ${\rm ad}(H)$, and let $\fru$ be the sum of positive eigenspaces of ${\rm ad}(H)$.
Then $\frq= \frl\oplus\fru$ be  a $\theta$-stable parabolic subalgebra of $\frg$ with Levi factor $\frl$ and nilpotent radical $\fru$. Set $L=N_{G}(\frq)$.

Let us arrange that $\Delta(\fru, \frh_f)\subseteq \Delta^{+}(\frg,\frh_f)$.
Set $\Delta^{+}(\frl, \frh_f)=\Delta(\frl,
\frh_f)\cap \Delta^{+}(\frg,\frh_f)$. Let
$\rho^{L}$  denote the half sum of roots in
$\Delta^{+}(\frl,\frh_f)$,  and denote by $\rho(\fru)$ (resp., $\rho(\fru\cap\frp)$) the half sum of roots in $\Delta(\fru,\frh_f)$ (resp., $\Delta(\fru\cap\frp,\frh_f)$). Then
\begin{equation}\label{relations}
\rho=\rho^{L}+\rho(\fru).
\end{equation}

Let $Z$ be an ($\frl$, $L\cap K$)-module. Cohomological induction functors attach to $Z$ certain ($\frg, K$)-modules $\caL_j(Z)$ and $\caR^j(Z)$, where $j$ is a nonnegative integer.
Suppose that $Z$ has infinitesimal character $\lambda_L\in \frh_f^*$. After \cite{KV}, we say that
$Z$ is {\it good} or {\it in good range} if
\begin{equation}\label{def-good}
\mathrm{Re}\,\langle \lambda_L +\rho(\fru),\, \alpha^\vee \rangle >
0, \quad \forall \alpha\in \Delta(\fru, \frh_f).
 \end{equation} We say that $Z$ is {\it weakly good} if
\begin{equation}\label{def-weakly-good}
\mathrm{Re}\, \langle \lambda_L +\rho(\fru),\, \alpha^\vee \rangle
\geq 0, \quad \forall \alpha\in \Delta(\fru, \frh_f).
\end{equation}

\begin{thm}\label{thm-Vogan-coho-ind}
{\rm (\cite{Vog84} Theorems 1.2 and 1.3, or \cite{KV} Theorems 0.50 and 0.51)}
Suppose the admissible
 ($\frl$, $L\cap K$)-module $Z$ is weakly good.  Then we have
\begin{itemize}
\item[(i)] $\caL_j(Z)=\caR^j(Z)=0$ for $j\neq S(:=\emph{\text{dim}}\,(\fru\cap\frk))$.
\item[(ii)] $\caL_S(Z)\cong\caR^S(Z)$ as ($\frg$, $K$)-modules.
\item[(iii)]  if $Z$ is irreducible, then $\caL_S(Z)$ is either zero or an
irreducible ($\frg$, $K$)-module with infinitesimal character $\lambda_L+\rho(\fru)$.
\item[(iv)]
if $Z$ is unitary, then $\caL_S(Z)$, if nonzero, is a unitary ($\frg$, $K$)-module.
\item[(v)] if $Z$ is in good range, then $\caL_S(Z)$ is nonzero, and it is unitary if and only if $Z$ is unitary.
\end{itemize}
\end{thm}

In the special case that $Z$ is a one-dimensional unitary character $\bbC_{\lambda}$, the module $\caL_S(Z)$ will be called an $A_{\frq}(\lambda)$ module. It is said to be \emph{fair} if
\begin{equation}\label{Aqlambda-fair}
\langle \lambda+\rho(\fru), \alpha \rangle > 0, \quad \forall \alpha\in \Delta(\fru, \frh_f),
\end{equation}
and to be \emph{weakly fair} if
\begin{equation}\label{Aqlambda-weakly-fair}
\langle \lambda+\rho(\fru), \alpha \rangle \geq 0, \quad \forall \alpha\in \Delta(\fru, \frh_f);
\end{equation}
Take $\Lambda\in\frh_f^*$ such that it is dominant for $\Delta^+(\frg, \frh_f)$.
We say that $\Lambda$ is \emph{real} if it belongs to $i\frt_{f, 0}^* +\fra_{f, 0}^*$,
and $\Lambda$ is {\it strongly regular} if
\begin{equation}\label{def-weakly-good}
\langle \Lambda-\rho,\, \alpha^\vee \rangle
\geq 0, \quad \forall \alpha\in \Delta^+(\frg, \frh_f).
\end{equation}

The following result says that a large part of $\widehat{G}$ consists of $A_{\frq}(\lambda)$ modules.

\begin{thm}\label{thm-SR} \emph{(Salamanca-Riba \cite{Sa})}
Let $\pi$ be an irreducible $(\frg, K)$-module with a strongly regular real infinitesimal character. If $\pi$ is unitary, then it is an $A_{\frq}(\lambda)$ module in the good range.
\end{thm}

Example \ref{exam-FS-Aqlambda} suggests that $A_{\frq}(\lambda)$ modules should continue to play an important role in the unitary dual for \emph{singular} infinitesimal characters.

\subsection{The \texttt{atlas} software}\label{sec-atlas}
Let us recall necessary notation from \cite{ALTV} regarding the Langlands parameters in the software \texttt{atlas} \cite{At}. The recent seminar \cite{Vog22} is also a good reference.
Let $H_{\bbC}$ be a \emph{maximal torus} of $G_{\bbC}$. That is, $H_{\bbC}$ is a maximal connected abelian subgroup of $G_{\bbC}$ consisting of diagonalizable matrices. Note that $H_{\bbC}$ is complex connected reductive algebraic. Its \emph{character lattice} is the group of algebraic homomorphisms
$$
X^*:={\rm Hom}_{\rm alg} (H_{\bbC}, \bbC^{\times}).
$$
Choose a Borel subgroup $B_{\bbC} \supset H_{\bbC}$.
In \texttt{atlas}, an irreducible $(\frg, K)$-module $\pi$ is parameterized by a \emph{final} parameter $p=(x, \lambda, \nu)$ via the Langlands classification \cite{ALTV}, where $\lambda \in X^*+\rho$, $\nu \in (X^*)^{-\theta}\otimes_{\bbZ}\bbC$, and $x$ is a KGB element. That is, $x$ is an $K_{\bbC}$-orbit in the Borel variety $G_{\bbC}/B_{\bbC}$. The infinitesimal character of $\pi$ is represented by
\begin{equation}\label{inf-char-atlas}
\frac{1}{2}(1+\theta)\lambda +\nu.
\end{equation}
Note that the Cartan involution $\theta$ now becomes $\theta_x$---the involution of $x$, which is given by the command \texttt{involution(x)} in \texttt{atlas}.

Among other things, Paul's lecture \cite{Paul} carefully explains how to do cohomological induction in \texttt{atlas}. In particular, the following canonical way is most important for us.

\begin{thm}\label{thm-Vogan} \emph{(Vogan \cite{Vog84})}
 Let $p=(x, \lambda, \nu)$ be the \texttt{atlas} parameter of an irreducible $(\frg, K)$-module $\pi$.
Let $S$ be the support of $x$, and $\frq(x)$ be the $\theta$-stable parabolic subalgebra given by the pair $(S, x)$, with Levi factor $L(x)$.
Then  $\pi$ is cohomologically induced, in the weakly good range,
from an irreducible $(\frl, L\cap K)$-module $\pi_{L(x)}$ with parameter $p_L=(y, \lambda-\rho(\fru),\nu)$, where $y$ is the KGB element of $L(x)$ corresponding to $x$.
\end{thm}

The \texttt{atlas} command for finding the $\pi_{L(x)}$ for any given $\pi$ is \texttt{reduce\_good\_range}.

\begin{example}\label{exam-atlas}
Let us look at the irreducible representations of $SL(2, \bbR)$ with infinitesimal character $\rho$.
As we shall see, there are four representations in total, among which three are unitary.
\begin{verbatim}
G:SL(2,R)
rho(G)
Value: [ 1 ]/1
set all=all_parameters_gamma(G,[1])
#all
Value: 4
void: for p in all do prints(p, "  ",is_unitary(p)) od
final parameter(x=2,lambda=[1]/1,nu=[1]/1)  true
final parameter(x=2,lambda=[2]/1,nu=[1]/1)  false
final parameter(x=1,lambda=[1]/1,nu=[0]/1)  true
final parameter(x=0,lambda=[1]/1,nu=[0]/1)  true
\end{verbatim}
The first one is the trivial representation. Now let us look at the third representation in \texttt{all}. By Theorem \ref{thm-SR}, it is an $A_{\frq}(\lambda)$ module.
\begin{verbatim}
set p=all[2]
p
Value: final parameter(x=1,lambda=[1]/1,nu=[0]/1)
set (Q,q)=reduce_good_range(p)
Q
Value: ([],KGB element #1)
Levi(Q)
Value: compact connected quasisplit real group with Lie algebra 'u(1)'
q
Value: final parameter(x=0,lambda=[0]/1,nu=[0]/1)
dimension(q)
Value: 1
goodness(q,G)
Value: "Good"
rho_u(Q)+infinitesimal_character(q)=infinitesimal_character(p)
Value: true
\end{verbatim}
We see that \texttt{p} is actually an $A_{\frb}(\lambda)$ module. Actually, it is a discrete series.\hfill\qed
\end{example}

\section{The structure of $E_{7(-25)}$}\label{sec-structure}

We will fix the group $G$ as \texttt{E7\_h} in \texttt{atlas} henceforth. This connected equal rank group has center $\bbZ/2\bbZ$. It is not simply connected. Note that $(G, K)$ is a Hermitian symmetric pair. The Lie algebra $\frg_0$ is labelled as EVII in \cite[Appendix C]{Kn}. We present
the Vogan diagram for $\frg_0$ in Fig.~\ref{Fig-EVII-Vogan}, where $\alpha_1=\frac{1}{2}(1, -1,-1,-1,-1,-1,-1,1)$, $\alpha_2=e_1+e_2$ and $\alpha_i=e_{i-1}-e_{i-2}$ for $3\leq i\leq 7$.  Let $\zeta_1, \dots, \zeta_7\in\frt_f^*$ be the corresponding fundamental weights for $\Delta^+(\frg, \frt_f)$, where $\frt_f\subset \frk$ is the fundamental Cartan subalgebra of $\frg$. The dual space $\frt_f^*$ will be identified with $\frt_f$ under the form $B(\cdot, \cdot)$. Put
\begin{equation}
\zeta:=\zeta_7=(0, 0, 0, 0, 0, 1, -\frac{1}{2}, \frac{1}{2}).
\end{equation}
Note that
$$
\rho=(0, 1, 2, 3, 4, 5, -\frac{17}{2}, \frac{17}{2}).
$$
We will use $\{\zeta_1, \dots, \zeta_7\}$ as a basis to express the \texttt{atlas} parameters $\lambda$, $\nu$ and the infinitesimal character $\Lambda$. More precisely, in such cases, $[a, b, c, d, e, f, g]$ will stand for the vector $a\zeta_1+\cdots+g \zeta_7$. For instance, the trivial representation of $E_{7(-25)}$ has infinitesimal character $\rho=[1,1,1,1,1,1,1]$.

\begin{figure}[H]
\centering
\scalebox{0.65}{\includegraphics{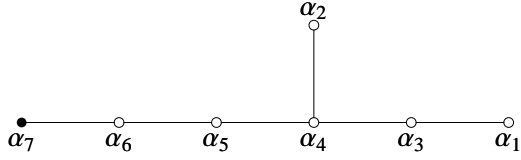}}
\caption{The Vogan diagram for $E_{7(-25)}$}
\label{Fig-EVII-Vogan}
\end{figure}

Denote by $\gamma_i=\alpha_{i}$ for $1\leq i\leq 6$.
 Let $\frt_f^-$ be the real linear span of $\gamma_1, \dots, \gamma_6$, which are the simple roots for $\Delta^+(\frk, \frt_f^-)$---the positive system for $\frk$ obtained from $\Delta^+(\frg, \frt_f)$ by restriction. We present the Dynkin diagram of $\Delta^+(\frk, \frt_f^-)$ in Fig.~\ref{Fig-EVII-Dynkin}.
Let $\varpi_1, \dots, \varpi_6\in (\frt_f^-)^*$ be the corresponding fundamental weights.
Note that $\bbR\zeta$ is the one-dimensional center of $\frk$,  that
$$
\frt_f=\frt_f^- \oplus \bbR\zeta,
$$
and that
$$
\rho_c=(0, 1, 2, 3, 4, -4, -4, 4).
$$
Moreover, we have the decomposition
\begin{equation}
\frp=\frp^+ \oplus \frp^-
\end{equation}
as $\frk$-modules, where $\frp^+$ has highest weight
$$
\beta:=2\alpha_1+ 2\alpha_2+ 3\alpha_3+ 4\alpha_4+ 3\alpha_5 +2\alpha_6 +\alpha_7=
(0, 0, 0, 0, 0, 0, -1, 1)
$$
and $\frp^-$ has highest weight $-\alpha_7$. Both of them have dimension $27$. Thus
$$
-\dim \frk + \dim\frp=-79+54=-25.
$$
This is how the number $-25$ enters the name $E_{7(-25)}$, see \cite{He}.

\begin{figure}[H]
\centering
\scalebox{0.6}{\includegraphics{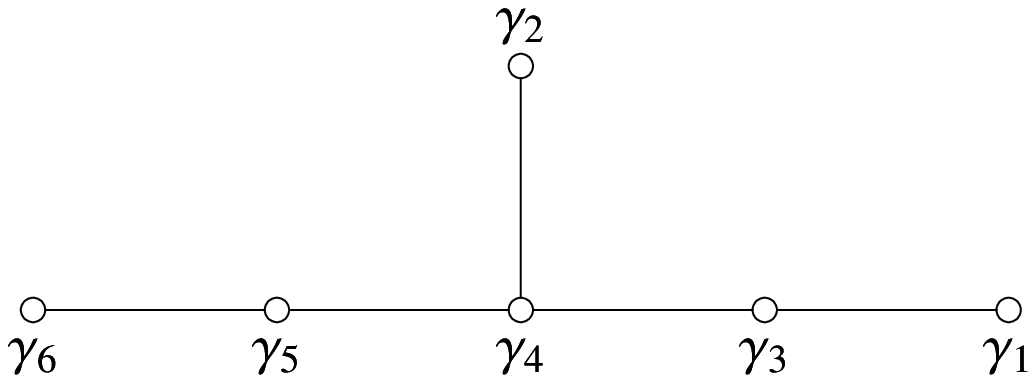}}
\caption{The Dynkin diagram for $\Delta^+(\frk, \frt_f^-)$}
\label{Fig-EVII-Dynkin}
\end{figure}

Let $E_{\mu}$ be the $\frk$-type with highest weight $\mu$.
We will use $\{\varpi_1, \dots, \varpi_6, \frac{1}{3}\zeta\}$ as a basis to express $\mu$.
Namely, in such a case, $[a, b, c, d, e, f, g]$ stands for the vector $a \varpi_1+b \varpi_2 + c \varpi_3+d \varpi_4 +e \varpi_5 +f \varpi_6+ \frac{g}{3}\zeta$. For instance,
\begin{equation}\label{beta-alpha7}
\beta=[1,0,0,0,0,0,2], \quad -\alpha_7=[0,0,0,0,0,1,-2],
\end{equation}
and $\rho_c=[1,1,1,1,1,1,0]$. The $\frk$-type $E_{[a, b, c, d, e, f, g]}$ has lowest weight $[-f, -b, -e, -d, -c, -a, g]$. Therefore, $E_{[f, b, e, d, c, a, -g]}$ is the \emph{contragredient $\frk$-type} of $E_{[a, b, c, d, e, f, g]}$.

For $a, b,c ,d ,e, f\in\bbZ_{\geq 0}$ and  $g\in\bbZ$, we have that $E_{[a, b, c, d, e, f, g]}$ is  a $K$-type if and only if
\begin{equation}\label{EIII-K-type}
-\frac{2 a}{3}-b-\frac{4 c}{3}-2 d-\frac{5 e}{3}-\frac{4 f}{3}+\frac{g}{3} \in\bbZ.
\end{equation}
Abusing the notation a bit, we may refer to a $\frk$-type (or $K$-type, $\widetilde{K}$-type) simply by its highest weight in terms of $\{\varpi_1, \dots, \varpi_6, \frac{1}{3}\zeta\}$.

\begin{lemma}\label{lemma-EVII-HP}
Let $\Lambda=a\zeta_1+b\zeta_2+c\zeta_3+d\zeta_4+e\zeta_5+f \zeta_6+g \zeta_7$ be the infinitesimal character of any Dirac series representation $\pi$ of $E_{7(-25)}$ which is dominant with respect to $\Delta^+(\frg, \frt_f)$. Then $a$, $b$, $c$, $d$, $e$, $f$, $g$ must be non-negative integers such that
\begin{equation}\label{HP-EVII-1}
a+c>0,   b+d>0,   c+d>0,   d+e>0,  e+f>0, f+g>0,
\end{equation}
and that
\begin{equation}\label{HP-EVII-2}
a+b+e>0, a+b+f>0, a+b+g>0, a+d+f>0, a+d+g>0,
\end{equation}
and that
\begin{equation}\label{HP-EVII-3}
a+e+g>0, b+c+e>0, b+c+f>0, b+c+g>0, c+e+g>0.
\end{equation}
\end{lemma}
\begin{proof}
Since the group \texttt{E7\_h} is linear, it follows from Remark 4.1 of \cite{D21} that $a$, $b$, $c$, $d$, $e$, $f$, $g$ must be non-negative integers.

Now if $a+c=0$, i.e., $a=c=0$,  a direct check says that for any $w\in W(\frg, \frt_f)^1$, at least one of the first six coordinates of $w\Lambda$ in terms of the basis $\{\varpi_1, \dots, \varpi_6, \frac{\zeta}{3}\}$ vanishes. Therefore,
$$
\{\mu-\rho_n^{(j)}\} + \rho_c =w \Lambda
$$
could not hold for any $K$-type $\mu$. This proves that $a+c>0$. Other inequalities can be similarly deduced.
\end{proof}

\section{Sharpening the the Helgason-Johnson bound for $E_{7(-25)}$}\label{sec-HJ}

Recall that $G$ is the group \texttt{E7\_h} in \texttt{atlas}. Moreover, we have fixed a Vogan diagram for $\frg_0$ in Fig.~\ref{Fig-EVII-Vogan}. By doing this, we have actually fixed a positive root system $\Delta^+(\frg, \frh_f)$. Recall that $\frh_f=\frt_f$. We start with $(\Delta^+)^{(0)}(\frg, \frt_f):=\Delta^+(\frg, \frh_f)$ which contains the $\Delta^+(\frk, \frt_f)$ corresponding to Fig.~\ref{Fig-EVII-Dynkin}. We fix this $\Delta^+(\frk, \frt_f)$ once for all.
There are $56$ ways of choosing  positive root systems for $\Delta(\frg, \frt_f)$ containing the $\Delta^+(\frk, \frt_f)$. Note that $|W(\frg, \frt_f)|/|W(\frk, \frt_f)|=56$.
We enumerate them as
$$
(\Delta^+)^{(j)}(\frg, \frt_f)=\Delta^+(\frk, \frt_f) \cup  (\Delta^+)^{(j)}(\frp, \frt_f), \quad 0\leq j\leq 55.
$$
Let us denote the half sum of roots of $(\Delta^+)^{(j)}(\frg, \frt_f)$ by $\rho^{(j)}$, and put $\rho_n^{(j)}:=\rho^{(j)} -\rho_c$. Then $\rho_n^{(j)}$ is the half sum of roots in $(\Delta^+)^{(j)}(\frp, \frt_f)$. Let $w^{(j)}$ be the unique element in $W(\frg, \frh_f)$ such that $w^{(j)} \rho^{(0)}=\rho^{(j)}$, and collect them as the set $W(\frg, \frt_f)^1$. Then $w^{(0)}=e$, and by a result of Kostant \cite{Ko}, the multiplication map induces a bijection from $W(\frg, \frt_f)^1 \times W(\frk, \frt_f)$ onto $W(\frg, \frt_f)$.
For any $0\leq j \leq 55$, $w^{(j)}\alpha_1, \dots, w^{(j)}\alpha_7$ are the simple roots of $(\Delta^+)^{(j)}(\frg, \frt_f)$, and $w^{(j)}\zeta_1, \dots, w^{(j)}\zeta_7$ are the corresponding fundamental weights.

Let us denote the dominant Weyl chamber for $(\Delta^+)^{(j)}(\frg, \frt_f)$ (resp., $\Delta^+(\frk, \frt_f)$) by $\caC^{(j)}$ (resp., $\caC$). Then
\begin{equation}\label{cone-deco}
\caC=\bigcup_{j=0}^{55}\caC^{(j)}.
\end{equation}
Note that $w^{(j)}\caC^{(0)}=\caC^{(j)}$ for $0\leq j\leq 55$. The convex hull formed by the $W(\frk, \frt_f)$ orbits of the points $\rho_n^{(j)}$, $0\leq j\leq 55$, is called the \emph{u-small convex hull}, and a $K$-type is said to be \emph{u-small} if its highest weight lies in this polyhedron. See Salamanca-Riba and Vogan \cite{SV}. Otherwise, we will call this $K$-type \emph{u-large}.

Consider an arbitrary $K$-type $E_{\mu}$. Choose $0\leq j\leq 55$ such that $\mu + 2\rho_c$ is dominant for $(\Delta^+)^{(j)}(\frg, \frt_f)$.
Put
\begin{equation}\label{lambda-a-mu}
\lambda_a(\mu):=P(\mu+2 \rho_c -\rho^{(j)}).
\end{equation}
Here for any vector $\eta\in i\frt_{f, 0}^*$,  $P(\eta)$ stands for the projection of $\eta$ to the cone $\caC^{(j)}$. Namely, $P(\eta)$ is the unique point in $\caC^{(j)}$ which is closest to $\eta$. It turns out that the vector $\lambda_a(\mu)$ is independent of the choice of an allowable $j$, and the \emph{lambda norm} of $\mu$ is defined as
\begin{equation}\label{lambda-norm}
\|\mu\|_{\rm lambda}:=\|\lambda_a(\mu)\|.
\end{equation}
The above geometric way of describing the lambda norm in  \cite{V81} is due to Carmona \cite{Ca}.
Let $\pi$ be an irreducible $(\frg, K)$ module. Then the $K$-type $\mu$ is called a \emph{lowest $K$-type} (LKT for short) of $\pi$ if $\mu$ shows up in $\pi$ and its lambda norm attains the minimum among all the $K$-types of $\pi$.

Given a LKT $\mu$ of $\pi$, the infinitesimal character $\Lambda$ of $\pi$ can be written as
\begin{equation}\label{inf-char}
\Lambda=(\lambda_a(\mu), \nu)\in \frh^*=\frt^*+\fra^*.
\end{equation}
As in \cite{V81}, $G(\lambda_a(\mu))$ is the isotropy group at $\lambda_a(\mu)$ for the $G$ action; $\frh$ is the complexified Lie algebra of a maximally split $\theta$-stable Cartan subgroup $H=TA$ of $G(\lambda_a(\mu))$. Note that $\nu$ in \eqref{inf-char} has the same norm as the $\nu$-part in the \texttt{atlas} parameter $p=(x, \lambda, \nu)$ of $\pi$. Abusing the notation a bit, we will not distinguish them.

Let $S_G$ be the irreducible module  of the Clifford algebra $C(\frp)$.
As a special case of Lemma 9.3.2 of \cite{Wa},
we have the following decomposition
\begin{equation}\label{spin-module}
S_G=\bigoplus_{j=0}^{55} E_{\rho_n^{(j)}}
\end{equation}
as $\frk$-modules. Now the spin norm introduced in \cite{D13} is
\begin{equation}\label{lambda-norm}
\|\mu\|_{\rm spin}:=\min_{0\leq j\leq 55} \|\{\mu-\rho_n^{(j)}\}  + \rho_c \|.
\end{equation}
Here $\{\mu-\rho_n^{(j)}\}$ is the unique dominant weight in the $W(\frk, \frt_f)$ orbit of $\mu-\rho_n^{(j)}$. We emphasize that $\{\mu-\rho_n^{(j)}\}$, $0\leq j\leq 55$, are precisely all the PRV components \cite{PRV} of the tensor product $E_{\mu}\otimes S_G$ as $\frk$-modules. A $K$-type $\delta$ of $\pi$ is called a \emph{spin lowest $K$-type} (spin LKT for short) of $\pi$ if its spin norm attains the minimum among all the $K$-types of $\pi$.

When $\pi$ is unitary, the case that we care the most, an equivalent way to formulate the Dirac inequality \eqref{Dirac-inequality} is
\begin{equation}\label{Dirac-inequality-eq}
\|\Lambda\| \leq \|\delta\|_{\rm spin},
\end{equation}
where $\delta$ is any $K$-type occurring in $\pi$. In view of the explanation after \eqref{Dirac-unitary}, $H_D(\pi)$ is non-zero if and only if \eqref{Dirac-inequality-eq} becomes equality on some $K$-types $\delta$ of $\pi$ (which then must be spin LKTs of $\pi$).

Now let us state the Helgason-Johnson bound and its improvements.

\begin{prop}\label{prop-HJ}
In the above setting, further assume that $\pi$ is unitary, then we have that
\begin{itemize}
\item[(a)] $\|\nu\|\leq \sqrt{\frac{399}{2}}=\|\rho\|$;
\item[(b)] $\|\nu\|\leq \sqrt{\frac{371}{2}}$;
\item[(c)] $\|\nu\|<\sqrt{94}$ if $\Lambda$ is integral and $\pi$ is not trivial or minimal.
\end{itemize}
\end{prop}

Item (a) above is due to Helgason-Johnson \cite{HJ} in 1969. Its slight extension (b) is given in \cite{D20} recently. In item (c), $\Lambda$ being integral means that it is an integer combination of $\zeta_1, \dots, \zeta_7$.

\begin{example}\label{exam-EVII-trivial-minimal}
The trivial representation of $G$ has \texttt{atlas} parameter:
\begin{verbatim}
final parameter(x=3016,lambda=[1,1,1,1,1,1,1]/1,nu=[4,0,0,0,0,4,1]/1)
\end{verbatim}
One computes that $\|\nu\|=\frac{371}{2}$, which attains the bound in item (b) of Proposition \ref{prop-HJ}.
Using \texttt{atlas}, one calculates that there are two irreducible representations with infinitesimal character $[1,1,1,0,1,1,1]$ and GK dimension $(\rho, \beta^\vee)=17$:
\begin{verbatim}
final parameter(x=2989,lambda=[3,2,2,-1,1,1,2]/1,nu=[8,5,5,-5,0,0,5]/2)
final parameter(x=2988,lambda=[3,2,2,-1,1,1,2]/1,nu=[8,5,5,-5,0,0,5]/2)
\end{verbatim}
Both of them are unitary with $\|\nu\|=\sqrt{97}$. Thus they \emph{exhaust} the  minimal representations of $G$. Both of them are Dirac series, and show up in Table \ref{table-EVII-1110111}.   \hfill\qed
\end{example}

To deduce item (c) of Proposition \ref{prop-HJ}, it suffices to prove the following
\begin{lemma}\label{lemma-HJ-EVII}
Assume that $\pi$ is a unitary $(\frg, K)$ module with integral infinitesimal character $\Lambda$ which is given by \eqref{inf-char}. If $\|\nu\|\geq \sqrt{94}$, then $\pi$ must be a minimal representation or the trivial representation.
\end{lemma}
\begin{proof}
Let $\mu$ be a LKT of $\pi$. Since $\pi$ is unitary, the Dirac inequality \eqref{Dirac-inequality-eq} guarantees that
$$
\|\Lambda\|^2=\|\lambda_a(\mu)\|^2+\|\nu\|^2\leq \|\mu\|^2_{\rm spin}.
$$
Therefore,
\begin{equation}\label{nu-bound-HS}
\|\nu\|^2\leq \|\mu\|_{\rm spin}^2 - \|\mu\|_{\rm lambda}^2.
\end{equation}
As computed in Section 6.2 of \cite{D20}, we have that $\max\{A_j\mid 0\leq j \leq 55\}=79$. We refer the reader to Section 3 of \cite{D20} for the precise definition of $A_j$. It follows that for any u-large $K$-type $\mu$, one has that
$$\|\mu\|_{\rm spin}^2 - \|\mu\|_{\rm lambda}^2\leq 79.
$$
Since $\|\nu\|\geq\sqrt{94}$, we conclude from \eqref{nu-bound-HS} that $\mu$ can \emph{not} be u-large.

There are $21294$ u-small $K$-types in total. Among them, only $71$ have the property that
$$
94\leq \|\mu\|_{\rm spin}^2 - \|\mu\|_{\rm lambda}^2.
$$
Let us collect these $71$ u-small $K$-types as \texttt{Certs}. Its elements are listed as follow:
\begin{align*}
&[0,0,0,0,0,0,0],  \\
&[0,2,0,0,0,0,0], \\
&[0,1,0,0,0,0,3m] \,(-3\leq m\leq 3), \\
&[0,0,0,1,0,0,3m] \,(-1\leq m\leq 1), \\
&[1,0,0,0,0,1,3m] \,(-2\leq m\leq 2),\\
&[0,0,0,0,0,0, 3m] \,(1\leq m\leq 4),  & [0,0,0,0,0,0, -3m] \,(1\leq m\leq 4),\\
&[0,0,0,0,0,3,3m] \,(2\leq m\leq 3),  &[3,0,0,0,0,0,-3m] \,(2\leq m\leq 3),\\
&[0,0,0,0,0,1,3m+1] \,(-4\leq m\leq 2), &[1,0,0,0,0,0,-3m-1] \,(-4\leq m\leq 2),\\
&[0,0,0,0,0,2,3m-1] \,(-2\leq m\leq 3), &[2,0,0,0,0,0,-3m+1] \,(-2\leq m\leq 3),\\
&[0,0,0,0,1,0,3m-1] \,(-2\leq m\leq 2),  &[0,0,1,0,0,0,-3m+1] \,(-2\leq m\leq 2),\\
&[0,1,0,0,0,1,3m+1] \,(-1\leq m\leq 1),  &[1,1,0,0,0,0,-3m-1] \,(-1\leq m\leq 1).
\end{align*}
Here we list a $K$-type and its contragredient in the same row.
Moreover, we compute that
$14\leq \|\lambda_a(\mu)\|^2\leq 49$ for any $\mu\in \texttt{Certs}$. Therefore,
\begin{equation}\label{can-Lambda}
14+ 94 \leq \|\Lambda\|^2=\|\lambda_a(\mu)\|^2+\|\nu\|^2\leq 49 +\frac{371}{2}.
\end{equation}
The right hand side above uses the proven Proposition \ref{prop-HJ}(b). There are $4676$ integral $\Lambda$s meeting the requirement \eqref{can-Lambda}. We collect them as $\Omega$.

Now a direct search using \texttt{atlas} says that there are $6474$ irreducible  representations $\pi$ such that $\Lambda\in \Omega$ and that $\pi$ has a LKT which is a member of \texttt{Certs}. Furthermore, only three of them turn out to be unitary. These three representations are explicitly presented in Example \ref{exam-EVII-trivial-minimal}: they are either trivial or minimal.
This finishes the proof.
\end{proof}

\section{Dirac series of $E_{7(-25)}$}

\subsection{FS-scattered representations of $E_{7(-25)}$}\label{sec-FS-EVII}

This subsection aims to seive out all the FS-scattered Dirac series representations for $E_{7(-25)}$ using the algorithm in \cite{D17}.

To achieve the goal, it suffices to consider all the infinitesimal characters $\Lambda=[a, b, c, d, e, f, g]$ such that
\begin{itemize}
\item[$\bullet$] $a$, $b$, $c$, $d$, $e$, $f$, $g$ are non-negative integers such that \eqref{HP-EVII-1}--\eqref{HP-EVII-3} hold;
\item[$\bullet$] $\min\{a, b, c, d, e, f, g\}=0$;
\item[$\bullet$] there exists a fully supported KGB element $x$ such that $\|\frac{\Lambda-\theta_x\Lambda}{2}\|< \sqrt{94}$.
\end{itemize}
Let us collect them as $\Phi$. Note that the first item is guaranteed by Lemma \ref{lemma-EVII-HP}. The second item uses Theorem \ref{thm-SR}.  The third item needs some explanation. Let $\pi$ be an irreducible unitary $(\frg, K)$ module with infinitesimal character $\Lambda$ which is dominant integral for $\Delta^{+}(\frg, \frt_f)$. Let $p=(x, \lambda, \nu)$ be the $\texttt{atlas}$ parameter for $\pi$. Then
\begin{equation}\label{nu}
\nu=\frac{\Lambda-\theta_x(\Lambda)}{2},
\end{equation}
where $\theta_x$ is \texttt{involution(x)} in \texttt{atlas}. Assume further that $\pi$ is \emph{neither trivial nor minimal}. Then by Proposition  \ref{prop-HJ}(c), we must have
$$
\|\nu\|< \sqrt{94}.
$$
The third item greatly reduces the cardinality of $\Phi$, which turns out to be $178192$.

Let us collect all the members of $\Phi$ whose largest coordinate equals to $i$ as $\Phi_i$. Then $\Phi$ is partitioned into $\Phi_1, \dots, \Phi_{13}$. Note that $\Phi_1$ has cardinality $23$, and that
\begin{center}
\begin{tabular}{c|c|c|c|c|c}
$\#\Phi_2$ & $\#\Phi_3$ & $\#\Phi_4$ & $\#\Phi_5$ & $\#\Phi_6$ & $\#\Phi_7$  \\
\hline
$921$  & $7817$ & $27246$ & $42088$ & $39685$ & $28107$\\
\hline
$\#\Phi_8$  & $\#\Phi_9$ & $\#\Phi_{10}$& $\#\Phi_{11}$ & $\#\Phi_{12}$ & $\#\Phi_{13}$  \\
\hline
$17649$ & $9042$ & $4022$ & $1359$ & $220$ & $13$
\end{tabular}
\end{center}
The elements of $\Phi_1$ are listed as follows:
\begin{align*}
&[0, 0, 1, 1, 1, 1, 1], [0, 1, 1, 0, 1, 1, 1], [0, 1, 1, 1, 0, 1, 1], [0, 1, 1, 1, 1, 0, 1], [0, 1, 1, 1, 1, 1, 0],\\
&[0, 1, 1, 1, 1, 1, 1], [1, 0, 0, 1, 1, 1, 1], [1, 0, 1, 1, 0, 1, 0], [1, 0, 1, 1, 0, 1, 1], [1, 0, 1, 1, 1, 0, 1]\\
&[1, 0, 1, 1, 1, 1, 0], [1, 0, 1, 1, 1, 1, 1], [1, 1, 0, 1, 0, 1, 1], [1, 1, 0, 1, 1, 0, 1], [1, 1, 0, 1, 1, 1, 0],\\
&[1, 1, 0, 1, 1, 1, 1], [1, 1, 1, 0, 1, 0, 1], [1, 1, 1, 0, 1, 1, 0], [1, 1, 1, 0, 1, 1, 1], [1, 1, 1, 1, 0, 1, 0],\\
&[1, 1, 1, 1, 0, 1, 1], [1, 1, 1, 1, 1, 0, 1], [1, 1, 1, 1, 1, 1, 0].
\end{align*}
A careful study of the irreducible unitary representations under the above $23$ infinitesimal characters leads us to Section \ref{sec-appendix}.

For the elements $\Lambda$ in $\Phi_2$, $\Phi_3$ up to $\Phi_{13}$, we use Proposition \ref{prop-HJ}(c) and \texttt{atlas} to verify that there is \emph{no} fully supported irreducible unitary representations with infinitesimal character $\Lambda$. Let us denote by $\Pi^{\rm u}_{\rm FS}(\Lambda)$ the set of all the fully supported irreducible unitary representations (up to equivalence) with infinitesimal character $\Lambda$.

\begin{example}\label{example-EVII-Phi8}
Let us consider the infinitesimal character $\Lambda=[1, 0, 1, 1, 1, 0, 8]$ in $\Phi_8$.
\begin{verbatim}
G:E7_h
set all=all_parameters_gamma(G, [1, 0, 1, 1, 1, 0, 8])
#all
Value: 525
set allFS=## for p in all do if #support(p)=7 then [p] else [] fi od
#allFS
Value: 246
\end{verbatim}
Therefore, there are $525$ irreducible representations under $\Lambda$, amon which $246$ are fully supported.

Now let us compare the original Helgason-Johnson bound and the sharpened one.
\begin{verbatim}
set HJ=94
set TgFWts=mat: [[0, 1, -1, 0, 0, 0, 0], [0, 1, 1, 0, 0, 0, 0],
[0, 1, 1, 2, 0, 0, 0], [0, 1, 1, 2, 2, 0, 0], [0, 1, 1, 2, 2, 2, 0],
[0, 1, 1, 2, 2, 2, 2], [-2, -2, -3, -4, -3, -2, -1], [2, 2, 3, 4, 3, 2, 1]]
set oldHJ=##for p in allFS do if (nu(p)*TgFWts)*(nu(p)*TgFWts)
<=4*399/2 then [p] else [] fi od
#oldHJ
Value: 218
set newHJ=##for p in allFS do if (nu(p)*TgFWts)*(nu(p)*TgFWts)
<4*HJ then [p] else [] fi od
#newHJ
Value: 29
\end{verbatim}
Therefore, there are $218$ fully  supported irreducible representations satisfying $\|\nu\|\leq \sqrt{\frac{399}{2}}$, while only $29$ of them satisfy $\|\nu\|< \sqrt{94}$.

Finally, let us test the unitarity of the representations in \texttt{newHJ}.
\begin{verbatim}
void: for p in newHJ do if is_unitary(p) then prints(p) fi od
\end{verbatim}
There is no output from \texttt{atlas}, meaning that no representation in \texttt{newHJ} is unitary. To sum up, we have that $\Pi^{\rm u}_{\rm FS}([1, 0, 1, 1, 1, 0, 8])=\emptyset$.
\hfill\qed
\end{example}
\begin{rmk}\label{rmk-example-EVI-Phi8}
Since \texttt{atlas} takes some time to testing unitarity, adopting the sharpened Helason-Johnson bound (thus reducing the number of representations for which we need to test their unitarity) saves us a lot of time.
\end{rmk}

\begin{example}\label{exam-EVII-1011010}
Consider the infinitesimal character $\Lambda=[1, 0, 1, 1, 0, 1, 0]$ for \texttt{E7\_h}.
It turns out that $\Pi^{\rm u}_{\rm FS}(\Lambda)$ consists of five representations, one of them is
\begin{verbatim}
set p=parameter(KGB(G,2969),[1,0,1,1,0,3,1],[0,0,0,0,0,4,0])
\end{verbatim}
Theorem \ref{thm-HP} guarantees that to study its Dirac cohomology, it suffices to look at its $K$-types up to the \texttt{atlas} height $248$. The \texttt{atlas} height of a $K$-type $\mu$ is given by
\begin{equation}\label{atlas-height}
\sum_{\alpha\in (\Delta^+)^{(j)}(\frg, \frt_f)}\langle \lambda_a(\mu), \alpha^\vee\rangle,
\end{equation}
where the $0\leq j\leq 55$ is chosen so that $\mu+2\rho_c$ is dominant with respect to $(\Delta^+)^{(j)}(\frg, \frt_f)$, and $\langle\cdot,\cdot\rangle$ is the natural pairing between roots and co-roots.
 Then the \texttt{atlas} command
\begin{verbatim}
print_branch_irr_long(p,KGB(G,55), 248)
\end{verbatim}
gives us $157$ such $K$-types in total, and the minimum spin norm of them is
$\sqrt{\frac{159}{2}}$, which is strictly larger than $\|\Lambda\|=\sqrt{78}$. Therefore, the unitary representation \texttt{p} has zero Dirac cohomology in view of the discussion following \eqref{Dirac-inequality-eq}.

The other four representations in $\Pi^{\rm u}_{\rm FS}(\Lambda)$ can be handled similarly. It turns out that they all have non-zero Dirac cohomology. We record them in Table \ref{table-EVII-1011010}, where the bolded \textbf{2949} under the column $\#x'$ means that there is another representation having KGB element $\#2949$, and having the same $\lambda$ and $\nu$ with that of $\#2950$.
\hfill\qed
\end{example}

All the other bolded entries under the columns $\#x'$ in the tables of Section \ref{sec-appendix} are interpreted similarly. In particular, counting the number of KGB elements there (being bolded or not) gives that there are $73$ \emph{non-trivial} FS-scattered representations in Section \ref{sec-appendix}.

\subsection{Counting the strings in $\widehat{E_{7(-25)}}^d$}\label{sec-EVII-esc}

Firstly, let us verify that Conjecture 2.6 of \cite{D21} and the binary condition holds for $E_{7(-25)}$.

\begin{example}
Consider the case that $\texttt{support(x)=[1, 2, 3, 4, 5, 6]}$. There are $236$ such KGB elements in total. We compute that there are $39002$ infinitesimal characters $\Lambda=[a, b, c, d, e, f, g]$ in total such that
\begin{itemize}
\item[$\bullet$] $b$, $c$, $d$, $e$, $f$, $g$ are non-negative integers, $a=0$ and that \eqref{HP-EVII-1}--\eqref{HP-EVII-3} hold;
\item[$\bullet$] there exists a  KGB element $x$ with support $[1, 2, 3, 4, 5, 6]$ such that $\|\frac{\Lambda-\theta_x\Lambda}{2}\|\leq \sqrt{94} $.
\end{itemize}
We exhaust all the irreducible unitary representations under these infinitesimal characters  with the above $236$ KGB elements. It turns out that such representations occur only when $b, c, d, e, f, g=0$ or $1$.  Then we check that each inducing module $\pi_{L(x)}$ (see Theorem \ref{thm-Vogan}) is indeed unitary. \hfill\qed
\end{example}

All the other non fully supported KGB elements are handled similarly. Thus Conjecture 2.6 of \cite{D21} and the binary condition hold for $E_{7(-25)}$.

Now we use the formula in Section 5 of \cite{D21} to figure out the number of strings in $\widehat{E_{7(-25)}}^d$. Recall that for any \emph{proper} subset $S$ of $\{1,2,3,4,5,6,7\}$, we use $N(S)$ to denote the number of Dirac series representations with infinitesimal character $\Lambda=\sum_{i=1}^{7} n_i \zeta_i$ such that $n_i$ is either $0$ or $1$ for each $i\in S$, and that $n_i=1$ for each $i\notin S$. For example, $N(\emptyset)$ counts the Dirac series representations with infinitesimal character $[1,1,1,1,1,1,1]$ and with KGB element $x$ such that $\texttt{support}(x)$ is empty. These representations are all tempered. It turns out that $N(\emptyset)=56$.
For each $0\leq i\leq 6$, we set
\begin{equation}\label{N-i}
N_i=\sum_{\#S=i}N(S).
\end{equation}
For instance, $N_0=N(\emptyset)=56$.

We compute that
\begin{align*}
&N([0,1,2,4,5,6])=0, \quad N([0,1,2,3,5,6])=4, \quad N([0,1,3,4,5,6])=2, \\
&N([0,1,2,3,4,6])=6, \quad N([0,2,3,4,5,6])=34, \quad N([1,2,3,4,5,6])=50,\\
&N([0,1,2,3,4,5])=62.
\end{align*}
In particular, it follows that $N_6=158$. We also compute that
$$
N_1=84, \quad N_2=102, \quad N_3=133, \quad N_4=164, \quad N_5=181.
$$
Therefore, the total number of strings for $E_{7(-25)}$ is equal to
$$
\sum_{i=0}^{6} N_i=878.
$$

Some auxiliary files have been built up to facilitate the classification of the Dirac series of $E_{7(-25)}$. They are available via the following link:
\begin{verbatim}
https://www.researchgate.net/publication/353352799_EVII-Files
\end{verbatim}

\section{Examples}

Keeping the notation for $E_{7(-25)}$ as in Section \ref{sec-structure}, we have a maximal $\theta$-stable parabolic subalgebra $\frq:=\frk+\frp^+$ of $\frg$. Let $E_{\mu}$ be the $K$-type with highest weight $\mu\in\frt_f^*$. Form the generalized Verma module
$$
N(\mu):=U(\frg)\otimes_{U(\frq)} E_{\mu}.
$$
Let $L(\mu)$ denote the irreducible quotient of $N(\mu)$. Those unitarizable $L(\mu)$, called (irreducible) \emph{unitary highest weight modules} in the literature, were classified in \cite{EHW, Ja}, and were known to be Dirac series representations \cite{HPP}. Indeed, by Proposition 3.7 of \cite{HPP}, the lowest $K$-type $E_{\mu}$ of $L(\mu)$ must contribute to $H_D(L(\mu))$. Thus $E_{\mu}$ must be one of the spin LKTs of $L(\mu)$. There is a similar story for irreducible unitary lowest weight modules. In the tables of Section \ref{sec-appendix}, we have marked out all the LKTs if they are simultaneously spin LKTs. Therefore, whenever the term ``LKT=" does not show up among the spin LKTs (see the $\#2960$ representation in Table \ref{table-EVII-0110111} for a specific instance), it means that the scattered representation is neither a highest weight module  nor a lowest weight module. Thus like the \texttt{E6\_h} case \cite{DDH}, the Dirac series of \texttt{E7\_h} go beyond the highest/lowest weight modules.

\begin{example}\label{exam-series-hwt}
By Theorem 13.4 (c) of \cite{EHW}, one finds that $L(z \zeta)$ is an irreducible unitary highest weight module of $E_{7(-25)}$ if and only if $z$ is an integer, and that $z=0, -4, -8$ or  $z\in (-\infty, -9]$. We locate these representations in Table \ref{table-hwt-zzeta}, where $g$ runs over the non-negative integers.

If the representation is non-trivial and fully supported,  the row ``Table lebel" gives the label of the table where we can find the representation. For instance, the $\#2365$ representations can be found in Table \ref{table-EVII-1110101}. The value $z=0$ produces the trivial representation (see Example \ref{exam-EVII-trivial-minimal} for its \texttt{atlas} parameter), while the values $z=-4$ and $-8$ are the two Wallach modules whose Dirac cohomology will be studied in Example \ref{exam-EVII-Wallach}.

Whenever $z\leq -12$, the representation $L(z\zeta)$ is not fully supported. Thus it will fit into a string, and does \emph{not} occur in the tables of Section \ref{sec-appendix}. The representation $L(z\zeta)$ is tempered whenever $z\leq -17$. \hfill\qed
\end{example}

\begin{table}[H]
\caption{Highest weight modules $L(z \zeta)$}
\begin{tabular}{r|c|c|c|c|c|c|}
$z$ &     $-11$ & $-10$ & $-9$ & $-8$ & $-4$ & $0$\\
\hline
$\#x$ &   $\#1224$ & $\#1604$ & $\#1975$ & $\#2365$& $\#2988$ & $\#3016$\\
\hline
Table label & \ref{table-EVII-1101011}   & \ref{table-EVII-1110101} & \ref{table-EVII-1011010} &  \ref{table-EVII-1110101} & \ref{table-EVII-1110111} &   \\
\hline
$z$ &  $-17-g$  & $-16$ & $-15$ & $-14$ & $-13$ & $-12$\\
\hline
$\#x$ &   $\#52$ & $\#66$ & $\#180$ & $\#301$& $\#443$ & $\#807$\\

\end{tabular}
\label{table-hwt-zzeta}
\end{table}

\begin{example}\label{exam-EVII-Wallach}
By Theorem 5.2 of \cite{EHW}, $E_{7(-25)}$ has two Wallach modules. Namely, $L(-4\zeta)$ and $L(-8\zeta)$.
See Example \ref{exam-series-hwt}.

The module $L(-4\zeta)$ has infinitesimal character $\Lambda=[1,1,1,0,1,1,1]$.
This is a minimal representation, and its $K$-types in terms of the basis $\{\varpi_1, \dots, \varpi_6, \frac{1}{3}\zeta\}$ are as follows:
$$
[0, 0, 0, 0, 0, 0, -12]+ n [0, 0, 0, 0, 0, 1, -2],
$$
where $n$ runs over non-negative integers.
It is easy to compute that the spin LKTs consist of those with $0\leq n\leq 5$, and they all have spin norm $\sqrt{\frac{231}{2}}$ which equals to $\|\Lambda\|$. Then we compute that $H_D(L(-4\zeta))$ consists of the following twelve $\widetilde{K}$-types without multiplicities:
\begin{align*}
&[1, 0, 0, 0, 0, 0, 11], [0, 0, 0, 0, 0, 1, -11], [2, 0, 0, 0, 0, 0, 1], [0, 0, 0, 0, 0, 2, -1],\\
&[0, 0, 0, 0, 1, 0, 5], [0, 0, 1, 0, 0, 0, -5], [0, 0, 0, 0, 0, 0, \pm 15],  [0, 1, 0, 0, 0, 0, \pm 9], [1, 0, 0, 0, 0, 1, \pm 3].
\end{align*}

On the other hand, the module $L(-8\zeta)$ has two spin LKTs:
$$
{\rm LKT}=[0,0,0,0,0,0,-24],\quad [1,0,0,0,0,0,-28],
$$
and infinitsimal character $\Lambda=[1,1,1,0,1,0,1]$. They have spin norm $\sqrt{\frac{159}{2}}$, which equals to $\|\Lambda\|$. Then we compute that $H_D(L(-8\zeta))$ consists of the $\widetilde{K}$-types $\zeta$ and $-\zeta$ without multiplicities.

For each of the above Wallach modules, the $\widetilde{K}$-types showing up in the Dirac chomology can be characterized as  members of the following set
\begin{equation}\label{eq-Dirac-wallach}
\{w\Lambda-\rho_c\mid w\in W(\frg, \frt_f)^1\}
\end{equation}
which are dominant for $\Delta^+(\frk, \frt_f)$.
\hfill\qed
\end{example}
\begin{rmk}\label{rmk-Wallach}
The formulation \eqref{eq-Dirac-wallach} is inspired by Theorem 1.3 of \cite{HPP}, which has computed  the Dirac cohomology of the Wallach modules for classical Hermitian symmetric real forms.
\end{rmk}

\begin{example}\label{exam-EVII-cancellation-condition}
Consider the $\#2950$ representation in Table \ref{table-EVII-1011010}. It has infinitesimal character $\Lambda=[1, 0, 1, 1, 0, 1, 0]$, which is conjugate to $\rho_c$ under the action of $W(\frg, \frt_f)$. It has LKT $\mu=[0, 0, 0, 0, 0, 0, 3]$ and three spin LKTs:
$$
\mu_1:=[0,0,0,0,0,1,25],\quad  \mu_2:=[4,0,0,0,0,1,9], \quad \mu_3:=[0,0,0,0,0,5,-7].
$$
Each $\mu_i$ contributes a trivial $\widetilde{K}$-type to the Dirac cohomology. We compute that
$$
B(\mu_1-\mu, \zeta)=11, \quad B(\mu_2-\mu, \zeta)=3, \quad B(\mu_3-\mu, \zeta)=5.
$$
These integers are of the same parity. Thus by Theorem 6.2 of \cite{DW21}, there is \emph{no} cancellation when passing from the Dirac cohomology to the Dirac index.

Therefore, up to a sign, the Dirac index of the representation is three copies of the trivial  $\widetilde{K}$-type. This can be checked directly by \texttt{atlas} as follows:
\begin{verbatim}
set p=parameter(KGB(G)[2950],[1,0,1,1,0,4,0],[0,0,0,0,0,4,0])
show_dirac_index(p)
coeff  lambda            higest weight    fund.wt.coords.       dim
3      [1,0,1,1,0,1,0]  [0,0,0,0,0,0,0 ]  [0,0,0,0,0,0]          1
\end{verbatim}
\hfill\qed
\end{example}

The following example is another illustration of cohomological induction.

\begin{example}\label{exam-FS-Aqlambda}
Let us demonstrate that the fifth entry of Table \ref{table-EVII-1110111} and the third entry of Table \ref{table-EVII-1101011} are both $A_{\frq}(\lambda)$ modules.
\begin{verbatim}
G:E7_h
set x=KGB(G,1550)
support(x)
Value: [0,2,3,4,5,6]
set Q=Parabolic:(support(x),x)
set L=Levi(Q)
set t=trivial(L)
theta_induce_irreducible(parameter(x(t),lambda(t)-[0,2,0,0,0,0,0],nu(t)),G)
Value:
1*parameter(x=1769,lambda=[1,3,2,-2,1,2,1]/1,nu=[1,5,2,-5,0,2,1]/1) [257]
theta_induce_irreducible(parameter(x(t),lambda(t)-[0,3,0,0,0,0,0],nu(t)),G)
Value:
1*parameter(x=1923,lambda=[1,1,-1,3,-2,2,1]/1,nu=[1,0,-3,5,-5,2,1]/1) [208]
\end{verbatim}
\hfill\qed
\end{example}

One can use the method in Example \ref{exam-FS-Aqlambda} to realize many FS-scattered representations as $A_{\frq}(\lambda)$ modules. But some FS-scattered representations of \texttt{E7\_h} are \emph{not} $A_{\frq}(\lambda)$ modules. To fill in this gap, we shall need special unipotent representations.

\section{Special unipotent representations}\label{sec-unip}
The special unipotent representations for \texttt{E7\_h} in the sense of \cite{BV} has been determined by Adams et al., see \cite{LSU}. Whenever such a representation is \emph{non-trivial} and is a Dirac series, we will mark it with a $\clubsuit$ in Section \ref{sec-appendix}. There are nine representations marked with  $\clubsuit$s in total, among which six are highest/lowest weight modules covered by Example \ref{exam-series-hwt}. It remains to analyze the $\#2950/\#2949$ in Table \ref{table-EVII-1011010} and the $\#2973$ in Table \ref{table-EVII-1110101}.

\begin{verbatim}
G:E7_h
set kgp=KGP(G,[0,1,2,3,4,5])
set P=kgp[3]
set L=Levi(P)
L
Value: connected real group with Lie algebra 'e6(so(10).u(1)).u(1)'
set t=trivial(L)
set tm9=parameter(x(t),lambda(t)-[0,0,0,0,0,0,9],nu(t))
goodness(tm9,G)
Value: "Weakly fair"
theta_induce_irreducible(tm9,G)
Value:
1*parameter(x=2950,lambda=[1,0,1,1,0,4,0]/1,nu=[0,0,0,0,0,4,0]/1) [100]
\end{verbatim}
We conclude that the $\#2950$ representation in Table \ref{table-EVII-1011010} is a weakly fair $A_{\frq}(\lambda)$ module. Then so is the $\#2949$ representation. On the other hand, the $\#2973$ in Table \ref{table-EVII-1110101} is \emph{not} a weakly fair $A_{\frq}(\lambda)$ module.

\section{Appendix}\label{sec-appendix}

This appendix presents all the $73$ \emph{non-trivial} FS-scattered Dirac series representations of $E_{7(-25)}$ according to their infinitesimal characters. The bolded KGB elements in the column $\#x^\prime$ are explained in Example \ref{exam-EVII-1011010}.

\begin{table}[H]
\centering
\caption{Infinitesimal character $[0,1,1,0,1,1,1]$}
\begin{tabular}{lccc}
 $\# x$ & $\lambda/\nu$ &  Spin LKTs  & $\# x^\prime$ \\
\hline
$2960$ & $[-2,2,4,-1,1,1,2]$  & $[0,0,0,0,0,1,16]$, $[0,0,0,0,0,5,2]$, & \textbf{2959}\\
 &  $[-3,\frac{5}{2},\frac{11}{2},-\frac{5}{2},0,0,\frac{5}{2}]$& $[4,0,0,0,0,1,18]$&\\
$915$ & $[0,2,3,-2,1,1,2]$  & $[3,1,0,0,0,1,16]$, $[3,0,0,0,1,0,20]$& \textbf{914}\\   & $[-2,\frac{5}{2},3,-3,0,0,\frac{5}{2}]$ &
\end{tabular}
\label{table-EVII-0110111}
\end{table}

\begin{table}[H]
\centering
\caption{Infinitesimal character $[1,0,0,1,1,1,1]$}
\begin{tabular}{lccc}
$\# x$ & $\lambda/\nu$ &  Spin LKTs  & $\# x^\prime$ \\
\hline
$2881$ & $[2,-1,-3,4,1,1,2]$  & ${\rm LKT}=[0,0,0,0,0,1,13]$, & \textbf{2880}\\
&  $[\frac{5}{2},-\frac{5}{2},-\frac{11}{2},\frac{11}{2},0,0,\frac{5}{2}]$& $[0,0,0,0,0,5,5]$, $[4,0,0,0,0,1,21]$&
\end{tabular}
\label{table-EVII-1001111}
\end{table}

\begin{table}[H]
\centering
\caption{Infinitesimal character $[1,0,1,1,0,1,0]$}
\begin{tabular}{lccc}
$\# x$ & $\lambda/\nu$ &  Spin LKTs  & $\# x^\prime$ \\
\hline
$2950_\clubsuit$ & $[1,0,1,1,0,4,0]$  & $[0,0,0,0,0,1,25]$, $[4,0,0,0,0,1,9]$, & $\textbf{2949}_\clubsuit$ \\
&  $[0,0,0,0,0,4,0]$& $[0,0,0,0,0,5,-7]$&\\
$1977_\clubsuit$ & $[1,-2,1,3,-2,3,0]$  & ${\rm LKT}=[0,0,0,0,0,0,27]$ & $\textbf{1975}_\clubsuit$ \\
&  $[0,-4,0,4,-4,4,0]$&
\end{tabular}
\label{table-EVII-1011010}
\end{table}

\begin{table}[H]
\centering
\caption{Infinitesimal character $[1,0,1,1,0,1,1]$}
\begin{tabular}{lccc}
$\# x$ & $\lambda/\nu$ &  Spin LKTs  & $\# x^\prime$ \\
\hline
$2684$ & $[1,-1,1,4,-3,2,1]$  & $[3,2,0,0,0,0,6]$, $[0,2,0,0,0,3,-6]$ \\
&  $[0,-4,2,5,-5,2,0]$&\\
$2017$ & $[1,-2,1,3,-2,3,1]$  & ${\rm LKT}=[1,0,0,0,0,0,26]$ &\textbf{2016}  \\
&  $[0,-\frac{9}{2},0,\frac{9}{2},-\frac{9}{2},4,1]$&		
\end{tabular}
\label{table-EVII-1011011}
\end{table}

\begin{table}[H]
\centering
\caption{Infinitesimal character $[1,1,0,1,0,1,1]$}
\begin{tabular}{lccc}
$\# x$ & $\lambda/\nu$ &  Spin LKTs  & $\# x^\prime$ \\
\hline
$2954$ & $[1,1,-1,4,-2,1,3]$  & $[0,0,0,0,0,1,19]$, $[0,0,0,0,0,5,-1]$, & \textbf{2953}\\
&  $[0,0,-2,5,-3,0,3]$& $[4,0,0,0,0,1,15]$&\\
$2127$ & $[4,2,-2,1,-1,3,1]$  & $[0,3,0,0,0,0,9]$, $[0,0,0,0,3,0,9]$& \textbf{2126}\\
&  $[5,\frac{3}{2},-\frac{7}{2},0,-\frac{3}{2},\frac{7}{2},0]$& \\
$1923$ & $[1,1,-1,3,-2,2,1]$  & ${\rm LKT}=[0,0,2,0,0,0,11]$,  & \textbf{1922}\\
&  $[1,0,-3,5,-5,2,1]$& $[0,0,2,0,0,2,7]$, $[0,0,0,2,0,0,15]$&\\
$1324$ & $[3,1,-2,2,-1,1,4]$  & $[0,0,0,0,0,1,31]$ & \textbf{1323}\\
&  $[3,0,-3,2,-2,0,4]$& \\
$1226$ & $[2,1,-1,2,-1,1,2]$  & ${\rm LKT}=[0,0,0,0,0,0,33]$& \textbf{1224}\\
&  $[3,0,-3,3,-3,0,3]$&
\end{tabular}
\label{table-EVII-1101011}
\end{table}

\begin{table}[H]
\centering
\caption{Infinitesimal character $[1,1,0,1,1,0,1]$}
\begin{tabular}{lccc}
$\# x$ & $\lambda/\nu$ &  Spin LKTs  & $\# x^\prime$ \\
\hline
$1957$ & $[4,2,-2,1,2,-2,3]$  & $[0,3,0,0,0,1,10]$, $[0,1,0,0,2,1,8]$& \textbf{1956}\\
&  $[5,\frac{3}{2},-\frac{7}{2},0,2,-\frac{7}{2},\frac{7}{2}]$& \\
$1524$ & $[3,1,-2,1,2,-1,4]$  & $[1,0,0,0,0,1,30]$ & \textbf{1523}\\   & $[\frac{7}{2},0,-\frac{7}{2},0,3,-3,4]$ &
\end{tabular}
\label{table-EVII-1101101}
\end{table}

\begin{table}[H]
\centering
\caption{Infinitesimal character $[1,1,0,1,1,1,1]$}
\begin{tabular}{lccc}
$\# x$ & $\lambda/\nu$ &  Spin LKTs  & $\# x^\prime$ \\
\hline
$2465$ & $[4,1,-3,2,1,2,1]$  & ${\rm LKT}=[0,4,0,0,0,0,0]$, \\
&  $[7,1,-7,2,0,2,0]$& $[1,4,0,0,0,0,2]$, $[0,4,0,0,0,1,-2]$&\\
$1713$ & $[2,1,-1,1,1,2,1]$  & ${\rm LKT}=[3,0,0,0,0,0,30]$, & \textbf{1712}\\
&   $[\frac{9}{2},0,-\frac{9}{2},0,0,4,1]$& $[2,0,1,0,0,0,32]$
\end{tabular}
\label{table-EVII-1101111}
\end{table}

\begin{table}[H]
\centering
\caption{Infinitesimal character $[1,1,1,0,1,0,1]$}
\begin{tabular}{lccc}
$\# x$ & $\lambda/\nu$ &  Spin LKTs  & $\# x^\prime$ \\
\hline
$2973_\clubsuit$ & $[1,2,1,-1,3,-1,4]$  & $[4,0,0,0,0,1,6]$, $[1,0,0,0,0,4,-6]$,\\
&  $[0,1,0,-1,4,-3,4]$& $[5,0,0,0,0,0,10]$, $[0,0,0,0,0,5,-10]$&\\
$2958$ & $[1,2,1,-1,4,-2,3]$  & $[0,0,0,0,0,1,22]$, $[0,0,0,0,0,5,-4],$& \textbf{2957}\\
&  $[0,1,0,-1,\frac{9}{2},-\frac{7}{2},\frac{7}{2}]$& $[4,0,0,0,0,1,12]$&\\
$2848$ & $[1,3,1,-2,5,-2,1]$  & $[3,1,0,0,0,1,4]$, $[1,1,0,0,0,3,-4],$\\
&  $[0,4,2,-4,5,-3,0]$& $[4,1,0,0,0,0,8]$, $[0,1,0,0,0,4,-8]$ & \\
$2366_\clubsuit$ & $[1,4,1,-3,4,0,1]$  & ${\rm LKT}=[0,0,0,0,0,0,24]$, & $\textbf{2365}_\clubsuit$\\
&  $[0,\frac{9}{2},0,-\frac{9}{2},\frac{9}{2},-\frac{1}{2},1]$& $[0,0,0,0,0,1,28]$ \\
$2299$ & $[1,3,1,0,1,-2,3]$  & $[0,0,0,0,3,0,0]$, $[1,0,0,0,2,1,4]$& \textbf{2298}\\
&  $[1,4,1,-1,0,-4,4]$& \\
$2233$ & $[1,4,1,-1,2,-3,5]$  & ${\rm LKT}=[0,0,0,0,0,1,22]$, & \textbf{2232}\\
&  $[0,4,0,-1,1,-4,5]$& $[0,0,0,0,1,0,26]$\\
$2131$ & $[2,2,3,-2,2,-2,3]$  & $[0,3,0,0,0,0,12]$, $[0,0,0,0,3,0,6]$& \textbf{2130}\\
&  $[\frac{3}{2},\frac{3}{2},\frac{7}{2},-\frac{7}{2},2,-\frac{7}{2},\frac{7}{2}]$& \\
$2081$ & $[1,2,2,-2,3,-2,4]$  & ${\rm LKT}=[0,0,0,0,1,0,20]$,& \textbf{2080}\\
&  $[0,\frac{3}{2},2,-\frac{7}{2},\frac{7}{2},-\frac{7}{2},5]$&  $[0,0,0,1,0,0,24]$ \\
$1824$ & $[1,1,4,-4,5,-2,3]$  & $[1,0,1,0,1,1,0]$, $[0,0,2,0,0,2,4],$\\
&  $[1,0,3,-4,4,-3,3]$&$[2,0,0,0,2,0,-4]$&\\
$1741$ & $[1,1,3,-2,3,-2,3]$  & $[0,0,1,0,1,2,2]$, $[1,0,0,0,2,1,-2]$& \textbf{1740}\\
&  $[1,0,3,-\frac{7}{2},\frac{7}{2},-\frac{7}{2},\frac{7}{2}]$& \\
$1669$ & $[1,1,3,-2,2,-1,4]$  & $[0,0,0,0,0,1,28]$& \textbf{1668}\\
&  $[0,0,\frac{7}{2},-\frac{7}{2},3,-3,4]$& \\
$1606$ & $[1,1,3,-2,3,-2,3]$  & ${\rm LKT}=[0,0,0,0,0,0,30]$& \textbf{1604}\\
&  $[0,0,\frac{7}{2},-\frac{7}{2},\frac{7}{2},-\frac{7}{2},\frac{7}{2}]$& \\
$1580$ & $[2,4,2,-3,2,-1,3]$  & ${\rm LKT}=[0,0,2,0,0,0,14]$, & \textbf{1579}\\
&  $[1,4,\frac{3}{2},-4,\frac{3}{2},-\frac{3}{2},\frac{5}{2}]$& $[1,1,1,0,0,0,18]$\\
$1025$ & $[2,2,1,-1,1,0,2]$  & $[0,1,0,0,0,4,-2]$, $[0,0,0,1,0,3,-6]$& \textbf{1023}\\
&  $[\frac{5}{2},\frac{5}{2},0,-\frac{5}{2},0,0,\frac{5}{2}]$& \\
$959$ & $[2,2,1,-2,3,-1,2]$  & ${\rm LKT}=[3,0,0,0,1,0,8]$, & \textbf{958}\\
&  $[2,2,0,-3,3,-2,1]$& $[3,1,0,0,0,1,10]$, $[3,0,0,1,0,0,6]$&
\end{tabular}
\label{table-EVII-1110101}
\end{table}

\begin{table}[H]
\centering
\caption{Infinitesimal character $[1,1,1,0,1,1,1]$}
\begin{tabular}{lccc}
$\# x$ & $\lambda/\nu$ &  Spin LKTs  & $\# x^\prime$ \\
\hline

$2989_\clubsuit$ & $[3,2,2,-1,1,1,2]$  & ${\rm LKT}=[0, 0, 0, 0, 0, 0, 12]$,  & $\textbf{2988}_\clubsuit$\\
&  $[4,\frac{5}{2},\frac{5}{2},-\frac{5}{2},0,0,\frac{5}{2}]$& $[0, 0, 0, 0, 0, 0, 12]+ n\beta$, $1\leq n\leq 5$ &\\

$2837$ & $[2,2,1,-2,3,1,2]$  & ${\rm LKT}=[0,0,0,0,0,2,14]$,  & \textbf{2836}\\
&  $[\frac{5}{2},3,0,-\frac{11}{2},\frac{11}{2},0,\frac{5}{2}]$& $[1,0,0,0,0,2,16]$, $[0,0,0,0,0,5,8]$, &\\ & &$[2,0,0,0,0,2,18]$, $[3,0,0,0,0,2,20]$ &\\
$2579$ & $[1,1,3,-1,1,1,1]$  & ${\rm LKT}=[0,3,0,0,0,0,0]$, \\
&  $[0,1,7,-5,0,2,0]$& $[1,3,0,0,0,0,2]$, $[0,3,0,0,0,1,-2]$,  &\\ & &$[2,3,0,0,0,0,4]$, $[0,3,0,0,0,2,-4]$ &\\
$1865$ & $[1,1,2,-1,1,2,1]$  & ${\rm LKT}=[2,0,0,0,0,0,28]$, & \textbf{1864}\\
&  $[0,0,\frac{9}{2},-\frac{9}{2},0,4,1]$& $[1,0,1,0,0,0,30]$, $[0,0,2,0,0,0,32]$&\\
$1769$ & $[1,3,2,-2,1,2,1]$  & ${\rm LKT}=[0,0,3,0,0,0,12]$, & \textbf{1768}\\
&  $[1,5,2,-5,0,2,1]$& $[0,0,3,0,0,1,10]$, $[0,0,2,1,0,0,14]$&\\
$1033$ & $[2,2,1,-1,1,1,2]$  & ${\rm LKT}=[0,1,0,0,0,0,36]$, & \textbf{1032}\\
&  $[3,3,0,-3,0,0,3]$& $[0,0,0,0,1,0,38]$
\end{tabular}
\label{table-EVII-1110111}
\end{table}

\begin{table}[H]
\centering
\caption{Infinitesimal character $[1,1,1,1,0,1,0]$}
\begin{tabular}{lccc}
$\# x$ & $\lambda/\nu$ &  Spin LKTs  & $\# x^\prime$ \\
\hline
$1438$ & $[1,1,2,1,-3,4,1]$  & $[2,0,1,0,0,3,2]$, $[3,0,0,0,1,2,-2]$\\
 &$[1,0,3,0,-4,4,-3]$ &
\end{tabular}
\label{table-EVII-1111010}
\end{table}

\begin{table}[H]
\centering
\caption{Infinitesimal character $[1,1,1,1,0,1,1]$}
\begin{tabular}{lccc}
$\# x$ & $\lambda/\nu$ &  Spin LKTs  & $\# x^\prime$ \\
\hline
$2768$ & $[2,2,1,1,-2,3,2]$  & ${\rm LKT}=[0,0,0,0,0,3,15]$, $[1,0,0,0,0,3,17],$ & \textbf{2767}\\
&  $[\frac{5}{2},3,0,0,-\frac{11}{2},\frac{11}{2},\frac{5}{2}]$& $[0,0,0,0,0,5,11]$, $[2,0,0,0,0,3,19]$&
\end{tabular}
\label{table-EVII-1111011}
\end{table}

\begin{table}[H]
\centering
\caption{Infinitesimal character $[1,1,1,1,1,0,1]$}
\begin{tabular}{lccc}
$\# x$ & $\lambda/\nu$ &  Spin LKTs  & $\# x^\prime$ \\
\hline
$2666$ & $[2,1,1,1,1,-1,3]$  & ${\rm LKT}=[0,0,0,0,0,4,16]$, & \textbf{2665}\\
&  $[\frac{5}{2},3,0,0,0,-\frac{11}{2},8]$&  $[0,0,0,0,0,5,14]$, $[1,0,0,0,0,4,18]$ &		
\end{tabular}
\label{table-EVII-1111101}
\end{table}

\centerline{\scshape Acknowledgements}
We are deeply grateful to the \texttt{atlas} mathematicians. We also thank Daniel Wong sincerely for helping us with Section \ref{sec-unip}.

\centerline{\scshape Funding}
Dong is supported by the National Natural Science Foundation of China (grant 12171344).

\end{document}